\documentclass[oneside,english]{amsart}
\usepackage[]{fontenc}
\usepackage[latin9]{inputenc}
\usepackage{color}
\usepackage{babel}
\usepackage{amsthm}
\usepackage{amstext}
\usepackage{amssymb}
\usepackage[numbers]{natbib}
\usepackage[unicode=true,pdfusetitle,
 bookmarks=true,bookmarksnumbered=false,bookmarksopen=false,
 breaklinks=false,pdfborder={0 0 1},backref=false,colorlinks=true]
 {hyperref}
\hypersetup{
 allcolors=blue}
\usepackage{breakurl}

\makeatletter
\numberwithin{equation}{section}
\numberwithin{figure}{section}
\theoremstyle{plain}
\newtheorem{thm}{\protect\theoremname}
  \theoremstyle{plain}
  \newtheorem{prop}[thm]{\protect\propositionname}
  \theoremstyle{plain}
  \newtheorem{lem}[thm]{\protect\lemmaname}
  \theoremstyle{plain}
  \newtheorem{cor}[thm]{\protect\corollaryname}
  \theoremstyle{remark}
  \newtheorem*{rem*}{\protect\remarkname}

\usepackage[a4paper]{geometry} 
\let\tempone\itemize
\let\temptwo\enditemize
\renewenvironment{itemize}{\tempone\addtolength{\itemsep}{0.3\baselineskip}}{\temptwo}

\makeatother

  \providecommand{\corollaryname}{Corollary}
  \providecommand{\lemmaname}{Lemma}
  \providecommand{\propositionname}{Proposition}
  \providecommand{\remarkname}{Remark}
\providecommand{\theoremname}{Theorem}

\begin{document}

\title{The Weil bound and non-exceptional permutation polynomials over finite
fields}

\author{Xiang Fan}

\address{School of Mathematics, Sun Yat-sen University, Guangzhou 510275,
China}
\begin{abstract}
A well-known result of von zur Gathen asserts that a non-exceptional
permutation polynomial of degree $n$ over $\mathbb{F}_{q}$ exists
only if $q<n^{4}$. With the help of the Weil bound for the number
of $\mathbb{F}_{q}$-points on an absolutely irreducible (possibly
singular) affine plane curve, Chahal and Ghorpade improved von zur
Gathen's proof to replace $n^{4}$ by a bound less than $n^{2}(n-2)^{2}$.
Also based on the Weil bound, we further refine the upper bound for
$q$ with respect to $n$, by a more concise and direct proof following
Wan's arguments.
\end{abstract}

\maketitle

\section{Introduction}

Let $\mathbb{F}_{q}$ denote the finite field of characteristic $p$
containing $q=p^{r}$ elements. A polynomial $f\in\mathbb{F}_{q}[x]$
is called a \emph{permutation polynomial} (PP) over $\mathbb{F}_{q}$
if the induced mapping $a\mapsto f(a)$ permutes $\mathbb{F}_{q}$.
Dating back to Hermite \citep{Hermite1863sur}  and Dickson \citep{Dickson1897analytic}
more than one hundred years ago, the study of PPs have been a hot
topic with interesting results occurring. However, the nontrivial
problem of determination or classification of PPs (of some prescribed
forms) are still challenging and interesting.

In 1897, Dickson \citep{Dickson1897analytic} classified all (normalized)
PPs of degree $d$ over $\mathbb{F}_{q}$ with $d\leqslant5$ for
all $q$, or $d=6$ for all odd $q$. (See also Shallue and Wanless
\citep{ShallueWanless2013permutation} for a reconfirmation of the
completeness of Dickson's classification. Note that the completeness
for $(d,q)=(6,3^{r})$ is up to linear transformations.) A remarkable
behavior in Dickson's classification is that, except for a few ``accidents''
over fields of low order, the PPs of a given degree are indeed exceptional
polynomials. By an \emph{exceptional polynomial} over $\mathbb{F}_{q}$
we mean a PP over $\mathbb{F}_{q}$ which is also a PP over $\mathbb{F}_{q^{m}}$
for infinitely many positive integers $m$. This behavior holds generally
for any degree, by the following result of Wan \citep{Wan1987conjecture}:
\begin{prop}
\label{prop:Cn} \emph{\citep[Theorem 2.4]{Wan1987conjecture}}\textbf{\emph{.}}
For every positive integer $n$, there exists a constant $C_{n}$,
depending only on $n$, such that all PPs of degree $n$ over $\mathbb{F}_{q}$
with $q>C_{n}$ are exceptional.
\end{prop}
An earlier version of Proposition \ref{prop:Cn} assuming $\mathrm{gcd}(q,n)=1$
was proved in Hayes \citep[Theorem 3.1]{Hayes1967geometric}. A quantitative
version was established by von zur Gathen \citep[Theorem 1]{Gathen1991values}
with $C_{n}$ replaced by $n^{4}$. Recently, Chahal and Ghorpade
\citep{ChahalGhorpade2018Carlitz} followed the arguments of \citep{Gathen1991values}
but utilized a precise version of the Weil bound for the number of
$\mathbb{F}_{q}$-points on an absolutely irreducible (possibly singular)
affine plane curve (the same as our quoted Lemma \ref{lem:Weil-bound}
with $\frac{\lfloor2\sqrt{q}\rfloor}{2}$ replaced by $\sqrt{q}$),
and ameliorated the bound $C_{n}$ to 
\[
\left(\frac{(n-2)(n-3)+\sqrt{(n-2)^{2}(n-3)^{2}+2(n^{2}-1)}}{2}\right)^{2},
\]
which is smaller than $n^{2}(n-2)^{2}$. Also based on the Weil bound,
our Theorem \ref{thm:bound} will refine the bound $C_{n}$ to 
\[
{\displaystyle \left(\frac{(n-2)(n-3)+\sqrt{(n-2)^{2}(n-3)^{2}+8n-12}}{2}\right)^{2}},
\]
which covers the main result of \citep{ChahalGhorpade2018Carlitz}
by a more concise and direct proof following the arguments of Wan
\citep[Theorem 2.4]{Wan1987conjecture} (which is also similar to
Lidl and Niederreiter \citep[\S7.29]{LidlNiederreiter1983finite}).

\section{Exceptional Polynomials}

Our definition of exceptional polynomials is slightly different from
what was used in \citep{Cohen1970distribution,LidlNiederreiter1983finite,Wan1987conjecture,Gathen1991values,ChahalGhorpade2018Carlitz}
and many other references. Note that for $f\in\mathbb{F}_{q}[x]$,
there exists a unique bivariate polynomial $f^{*}(x,y)$ in $\mathbb{F}_{q}[x,y]$
(the ring of polynomials in two indeterminates $x$ and $y$ with
coefficients in $\mathbb{F}_{q}$) such that 
\[
f(x)-f(y)=(x-y)f^{*}(x,y)\quad\text{in }\mathbb{F}_{q}[x,y].
\]
They called $f$ an exceptional polynomial over $\mathbb{F}_{q}$
if the corresponding $f^{*}(x,y)$ has no absolutely irreducible (i.e.
irreducible over an algebraic closure of $\mathbb{F}_{q}$) factor
in $\mathbb{F}_{q}[x,y]$. In other words, their definition means
that each irreducible factor $g(x,y)\in\mathbb{F}_{q}[x,y]$ of $f^{*}(x,y)$
is reducible in $\mathbb{F}_{q^{k}}[x,y]$ for some integer $k>1$.

The difference between these two definitions can be clarified as follows
in our language. Recall that $p$ is the characteristic of $\mathbb{F}_{q}$.
Let $f'$ denote the formal derivative of a polynomial $f\in\mathbb{F}_{q}[x]$.
Note that $f'=0$ if and only if $f(x)=g(x)^{p}$ for some $g\in\mathbb{F}_{q}[x]$.
For a non-constant polynomial $f\in\mathbb{F}_{q}[x]$, repeating
this process if necessary, we can write $f(x)=f_{\circ}(x)^{p^{t}}$
for an integer $t\geqslant0$ and $f_{\circ}\in\mathbb{F}_{q}[x]$
with $f_{\circ}'\neq0$. Clearly for any $\mathbb{F}_{q^{m}}$, $f$
is a PP over $\mathbb{F}_{q^{m}}$ if and only if $f_{\circ}$ is
a PP over $\mathbb{F}_{q^{m}}$. Indeed, the following three statements
are equivalent for any non-constant $f\in\mathbb{F}_{q}[x]$: 

$(1)$ $f$ is an exceptional polynomial over $\mathbb{F}_{q}$;

$(2)$ $f_{\circ}$ is an exceptional polynomial over $\mathbb{F}_{q}$;

$(3)$ $f_{\circ}^{*}(x,y)$ has no absolutely irreducible factor
in $\mathbb{F}_{q}[x,y]$.

$(4)$ Every absolutely irreducible factor of $f(x)-f(y)$ in $\mathbb{F}_{q}[x,y]$
is a constant times $x-y$.

Here $(1)\Leftrightarrow(2)$ by our definition, $(2)\Leftrightarrow(3)$
by Cohen \citep[Theorem 4, 5]{Cohen1970distribution} provided $f_{\circ}'\neq0$,
 and $(3)\Leftrightarrow(4)$ by 
\[
(x-y)f^{*}(x,y)=f_{\circ}(x)^{p^{t}}-f_{\circ}(y)^{p^{t}}=(f_{\circ}(x)-f_{\circ}(y))^{p^{t}}=(x-y)^{p^{t}}f_{\circ}^{*}(x,y)^{p^{t}}.
\]

\section{The Weil Bound}

Andr\'{e} Weil \citep{Weil1948} provided a bound for the $\mathbb{F}_{q}$-points
of smooth projective curves defined over $\mathbb{F}_{q}$. Leep and
Yeomans \citep{LeepYeomans1994number}, Aubry and Perret \citep{AubryPerret1996Weil},
and Bach \citep{Bach1996Weil} established similar bounds for singular
curves. For our use, the following Lemma \ref{lem:Weil-bound} available
by \citep[Corollary 2.b]{LeepYeomans1994number} gives a precise version
of the Weil bound for the number of $\mathbb{F}_{q}$-points on an
absolutely irreducible (possibly singular) affine plane curve over
$\mathbb{F}_{q}$.

Let $|S|$ stand for the cardinality of a set $S$, and $\lfloor t\rfloor$
for the greatest integer $\leqslant t\in\mathbb{R}$. For a bivariate
polynomial $\varphi(x,y)\in\mathbb{F}_{q}[x,y]$, write the set of
its zeros over $\mathbb{F}_{q}$ as 
\[
Z_{q}(\varphi)=\{(a,b)\in\mathbb{F}_{q}\times\mathbb{F}_{q}:\varphi(a,b)=0\}.
\]
By definition, $f\in\mathbb{F}_{q}[x]$ is a PP over $\mathbb{F}_{q}$
if and only if $Z_{q}(f^{*})\subseteq\{(a,a):a\in\mathbb{F}_{q}\}$.
\begin{lem}
\label{lem:Weil-bound} \emph{\citep[Corollary 2.b]{LeepYeomans1994number}}
If $\varphi(x,y)\in\mathbb{F}_{q}[x,y]$ is an absolutely irreducible
bivariate polynomial of degree $d\geqslant1$, then 
\[
q+1-(d-1)(d-2)\frac{\lfloor2\sqrt{q}\rfloor}{2}-d\leqslant|Z_{q}(\varphi)|\leqslant q+1+(d-1)(d-2)\frac{\lfloor2\sqrt{q}\rfloor}{2}.
\]
\end{lem}
\begin{cor}
\label{cor:qd} Let $\varphi(x,y)\in\mathbb{F}_{q}[x,y]$ be an absolutely
irreducible bivariate polynomial of degree $d\geqslant1$, such that
$\varphi(x,x)\neq0$ in $\mathbb{F}_{q}[x]$ and $Z_{q}(\varphi)\subseteq\{(a,a):a\in\mathbb{F}_{q}\}$.
Then 
\[
q+1\leqslant(d-1)(d-2)\frac{\lfloor2\sqrt{q}\rfloor}{2}+2d,
\]
and in particular, $\sqrt{q}\leqslant\dfrac{(d-1)(d-2)+\sqrt{(d-1)^{2}(d-2)^{2}+8d-4}}{2}$.\end{cor}
\begin{proof}
Clearly, $Z_{q}(\varphi)\subseteq\{(a,a):\varphi(a,a)=0,\ a\in\mathbb{F}_{q}\}$.
As $\varphi(x,x)\neq0$ in $\mathbb{F}_{q}[x]$, and $0\leqslant\deg(\varphi(x,x))\leqslant\deg(\varphi(x,y))=d$,
we have 
\[
|Z_{q}(\varphi)|\leqslant|\{a\in\mathbb{F}_{q}:\varphi(a,a)=0\}|\leqslant\deg(\varphi(x,x))\leqslant d.
\]
By Lemma \ref{lem:Weil-bound} for absolutely irreducible $\varphi(x,y)$,
\[
q+1-(d-1)(d-2)\frac{\lfloor2\sqrt{q}\rfloor}{2}-d\leqslant|Z_{q}(\varphi)|\leqslant d.
\]
Replacing $\frac{\lfloor2\sqrt{q}\rfloor}{2}$ by $\sqrt{q}$, we
can deduce the latter inequality.\end{proof}
\begin{rem*}
Corollary \ref{cor:qd} and its proof are quite similar to Hayes \citep[Theorem 2.1]{Hayes1967geometric},
which only used Lang-Weil inequality and gained an abstract bound
for $q$. Also, Aubry and Perret \citep[Lemma 3.2]{AubryPerret1996Weil}
seemed to give the same bound on $q$ as Corollary \ref{cor:qd} by
the same arguments, but unfortunately miscalculated $\sqrt{(d-1)^{2}(d-2)^{2}+8d-4}$
as $\sqrt{(d-1)(d-2)+8d-4}=\sqrt{d^{2}+5d-2}$.
\end{rem*}

\section{The Main Result}
\begin{thm}
\label{thm:bound} If there exists a non-exceptional PP $f$ of degree
$n$ over $\mathbb{F}_{q}$, then  
\[
q+1\leqslant(n-2)(n-3)\dfrac{\lfloor2\sqrt{q}\rfloor}{2}+2(n-1),
\]
and in particular $q\leqslant\left\lfloor \left(\dfrac{(n-2)(n-3)+\sqrt{(n-2)^{2}(n-3)^{2}+8n-12}}{2}\right)^{2}\right\rfloor $.\end{thm}
\begin{proof}
Write $f(x)=f_{\circ}(x)^{p^{t}}$ with an integer $t\geqslant0$
and $f_{\circ}\in\mathbb{F}_{q}[x]$ with nonzero $f_{\circ}'$. Then
$f_{\circ}$ is also a non-exceptional PP over $\mathbb{F}_{q}$,
and $f_{\circ}^{*}(x,y)$ has an absolutely irreducible factor in
$\mathbb{F}_{q}[x,y]$ denoted by $\varphi(x,y)$. As $f_{\circ}^{*}(x,x)=f_{\circ}'(x)$
is nonzero, its factor $\varphi(x,x)$ is also nonzero. For the PP
$f_{\circ}$ over $\mathbb{F}_{q}$, clearly 
\[
Z_{q}(\varphi)\subseteq Z_{q}(f_{\circ}^{*})\subseteq\{(a,a):a\in\mathbb{F}_{q}\}.
\]
Let $d=\deg(\varphi(x,y))\leqslant\deg(f_{\circ}^{*}(x,y))=\deg(f_{\circ})-1\leqslant n-1$.
Note that all PPs of degree $\leqslant3$ are exceptional (cf. Dickson's
classification \citep{Dickson1897analytic}), so $n\geqslant4$. By
Corollary \ref{cor:qd}, 
\[
q+1\leqslant(d-1)(d-2)\frac{\lfloor2\sqrt{q}\rfloor}{2}+2d\leqslant(n-2)(n-3)\frac{\lfloor2\sqrt{q}\rfloor}{2}+2(n-1).
\]
In particular, $q+1\leqslant(n-2)(n-3)\sqrt{q}+2(n-1)$, which implies
the latter inequality.\end{proof}
\begin{rem*}
This proof totally follows the arguments of Wan \citep[Theorem 2.4]{Wan1987conjecture},
which is also similar to Lidl and Niederreiter \citep[\S7.29]{LidlNiederreiter1983finite}.
Our only improvement is replacing the abstract bound that they used
(indeed \citep[Theorem 2.1]{Hayes1967geometric} or \citep[\S7.28]{LidlNiederreiter1983finite})
by the explicit bound or inequality in Corollary \ref{cor:qd}.\end{rem*}
\begin{cor}
\label{cor:409} If there exists a non-exceptional PP of degree $7$
over $\mathbb{F}_{q}$, then $q\leqslant409$.\end{cor}
\begin{proof}
By Theorem \ref{thm:bound}, $q\leqslant\left\lfloor \left(\frac{20+\sqrt{20^{2}+8\cdot7-12}}{2}\right)^{2}\right\rfloor =421$,
and $q+1\leqslant10\lfloor2\sqrt{q}\rfloor+12$, which does not hold
for $q\in\{421,419\}$. The greatest prime power below $419$ is $409$.
\end{proof}
All exceptional polynomials of degree 7 over $\mathbb{F}_{q}$ have
been classified by M\"{u}ller \citep{Muller1997Weil} and Fried,
Guralnick, Saxl \citep{FriedGuralnickSaxl1993Schur}. Moreover, Li,
Chandler and Xiang \citep{LiChandlerXiang2010permutation} obtained
the classification of all PPs of degree 6 and 7 over $\mathbb{F}_{2^{t}}$
for any $t\geqslant3$. Therefore, to classify all PPs of degree $7$
over $\mathbb{F}_{q}$, it suffices to consider the odd prime powers
$q$ with $7<q\leqslant409$. We have done this with the help of the
SageMath software running on a personal computer, and written down
all PPs of degree $7$ up to linear transformations in our preprint
\citep{Fan2019PP7}. In particular, \citep{Fan2019PP7} shows that
a non-exceptional PP of degree $7$ over $\mathbb{F}_{q}$ (with $q>7$)
exists if and only if
\[
q\in\{9,11,13,17,19,23,25,27,31,49\}.
\]
In particular, given a root $e$ of $x^{2}+6x+3$ in $\mathbb{F}_{7^{2}}$,
$x^{7}+ex^{5}+e^{18}x^{3}+e^{35}x$ is a non-exceptional PP over $\mathbb{F}_{7^{2}}$.
Therefore, for $n=7$, the sharp bound $C_{7}$ in Proposition \ref{prop:Cn}
is $7^{2}$.

A natural question asks for the sharp bound $C_{n}$ for $n>7$. Note
that $n^{2}$ is not a general answer by the non-exceptional PP $f(x)=x^{10}+3x$
over $\mathbb{F}_{7^{3}}$ from Masuda and Zieve \citep{MasudaZieve2009binomial}.
However, it would be interesting to know whether $C_{n}$ can be replaced
by some quadratic in $n$, in the spirit of Mullen \citep[Conjecture 2.5]{Mullen1993PP}.

\end{document}